\documentclass{amsart}  
\usepackage{amsfonts}
\usepackage{amsthm}
\usepackage{amssymb}
\usepackage{amsmath}
\usepackage[T1]{fontenc}
\usepackage[utf8]{inputenc}
\usepackage{mathrsfs}
\usepackage{verbatim}
\usepackage{tikz}
\theoremstyle{plain} 
\newtheorem{Lemma}{Lemma}[section] \newtheorem{Thm}[Lemma]{Theorem} \newtheorem{Prop}[Lemma]{Proposition}\newtheorem{Cor}[Lemma]{Corollary}  
\theoremstyle{definition} \newtheorem{Defn}[Lemma]{Definition} \newtheorem{Ex}[Lemma]{Example} 
\theoremstyle{remark} \newtheorem{Rem}[Lemma]{Remark}  
\numberwithin{equation}{section}

\newcommand{\ZZ}{\mathbb{Z}}
\newcommand{\NN}{\mathbb{N}}
\newcommand{\dd}{\partial}
\newcommand{\inj}{\hookrightarrow}

\newcommand{\sse}{\subset}

\newcommand{\ess}{\supset}

\newcommand{\p}{\mathfrak{p}}

\begin{document}

\author{Ketil Tveiten}
\title{Two-sided ideals in the ring of differential operators on a Stanley-Reisner ring}
\address{Ketil Tveiten\\Matematiska Institutionen\\ Stockholms Universitet\\ 106 91 Stockholm.}
\email{ktveiten@math.su.se}
\begin{abstract}
Let $R$ be a Stanley-Reisner ring (that is, a reduced monomial ring) with coefficients in a domain $k$, and $K$ its associated simplicial complex. Also let $D_k(R)$ be the ring of $k$-linear differential operators on $R$. We give two different descriptions of the two-sided ideal structure of $D_k(R)$ as being in bijection with certain well-known subcomplexes of $K$; one based on explicit computation in the Weyl algebra, valid in any characteristic, and one valid in characteristic $p$ based on the Frobenius splitting of $R$. A result of Traves \cite{Traves} on the $D_k(R)$-module structure of $R$ is also given a new proof and different interpretation using these techniques.
\end{abstract}

\maketitle

%
\section{Introduction}\label{Sec:Introduction}
%

Rings of $k$-linear differential operators $D_k(R)$ on a $k$-algebra $R$ are generally difficult to study, even when the base ring $R$ is well-behaved. Some descriptions of $D_k(R)$ are given in e.g. \cite{Musson} for the case of toric varieties, \cite{Bavula-0} and \cite{Bavula-p} for general smooth affine varieties (in zero and prime characteristic respectively), and \cite{Traves}, \cite{Tripp} and \cite{Eriksson} for Stanley-Reisner rings. 
 Some criteria for simplicity of $D_k(R)$ exist (see \cite{Smith-vdB} and \cite{Saito} among others), and the study of their left and right ideals, through the theory of $D$-modules, is well developed. 

When $D_k(R)$ is not simple, however, it is an interesting problem to give a description of its \emph{two-sided} ideals; the purpose of this paper is to do this for the case of Stanley-Reisner rings. Every Stanley-Reisner ring is the face ring $R_K$ of a simplicial complex $K$, and we will give two different descriptions of the two-sided ideal structure of $R$ in terms of the combinatorial structure of $K$; namely the lattice of ideals is in a certain sense determined by the poset of subcomplexes of $K$ that are \emph{stars} of some face of $K$. The first description is based on explicit computations with monomials in the Weyl algebra, and the second (valid only in prime characteristic) takes advantage of the Frobenius splitting of $R$. 

%
\section{Some preliminaries}\label{Sec:prelims}
%

Let us fix some notation. Throughout, $k$ is a commutative domain. $K$ will denote an abstract simplicial complex on vertices $x_1,\ldots,x_n$; we will not distinguish between $K$ as an abstract simplicial complex and its topological realization. In the corresponding face rings (see \ref{Defn:SR-ring}) the indeterminate corresponding to a vertex $x_i$ will also be named $x_i$ to avoid notational clutter. 
Elements of $K$ will be referred to as \emph{simplices} or \emph{faces}. For a face $\sigma\in K$, we let $x_\sigma:=\prod_{x_i\in\sigma}x_i$.
$R$ will always mean a face ring $R_K$ for a simplicial complex $K$. We use standard multiindex notation: $x^a$ denotes $x_1^{a_1}\cdots x_n^{a_n}$, and $|a|=a_1+\cdots a_n$.

We briefly recall for the benefit of the reader some basics of Stanley-Reisner rings, omitting the proofs.

\begin{Defn}\label{Defn:SR-ring}
  Let $K$ be an abstract simplicial complex on vertices $x_1,\ldots,x_n$. The \emph{Stanley-Reisner ring}, or \emph{face ring}, of $K$ with coefficients in $k$ is the ring $R_K=k[x_1,\ldots,x_n]/I_K$, where $I_K=\langle x_{i_1}\cdots x_{i_r}|\{x_{i_1},\ldots,x_{i_r}\}\not\in K\rangle$ is the ideal of square-free monomials corresponding to the non-faces of $K$, called the \emph{face ideal of $K$}. 
\end{Defn}

Geometrically, $R_K$ is the coordinate ring of the cone on $K$, so $\dim R_K=\dim K+1$. Accordingly, when we talk about support of elements, we will refer to faces of $K$ when strictly speaking we mean the cones on these faces. If $K=\Delta_n$ is a simplex, $I_K$ is the zero ideal, and $R_K$ is the polynomial ring in $n$ variables. If $K=K'\ast K''$ is the simplicial join of complexes $K'$ and $K''$, then $R_K\simeq R_{K'}\otimes_k R_{K''}$. Face rings are exactly the reduced monomial rings, i.e. quotients of polynomial rings by square-free monomial ideals. 

Given a simplicial complex $K$, we will have use for a well-known class of subsets of $K$:

\begin{Defn}\label{Defn:star}
  Let $\sigma\in K$ be a face. The \emph{closed star} of $\sigma$ in $K$ is the subcomplex 
\[
st(\sigma,K):=\{\tau\in K|\tau\cup\sigma\in K\}.
\]
The \emph{open star} of $\sigma$ in $K$ is the set 
\[
st(\sigma,K)^\circ:=\{\tau\in K|\sigma\cup\tau\in K \wedge\sigma\cap \tau\neq \varnothing\};
\]
$st(\sigma,K)^\circ$ is the interior of $st(\sigma,K)$ in $K$, and $st(\sigma,K)$ is the closure of $st(\sigma,K)^\circ$ in $K$. 
The \emph{open complement} of $st(\sigma,K)$ is the set (not usually a subcomplex) 
\[
U_\sigma(K)=K\setminus st(\sigma,K) =\{\tau\in K| \tau\cup\sigma\not\in K\}.
\]
\end{Defn}

Stars are important because the support of a principal monomial ideal of $R_K$, considered as an $R_K$-module, is exactly equal to the open star of some face, and the closed star is the smallest subcomplex containing it. For the remainder, we will take \emph{star} to mean \emph{closed star}.
We will not have much need of comparing stars associated to different subcomplexes and so will often write simply $st(\sigma),U_\sigma$ if no confusion is likely to result. For completeness, we repeat a few simple facts:

\begin{Lemma}\label{Lemma:stars}
  \begin{itemize}
  \item[(i)] If $\sigma\sse\tau$ are faces in $K$, $st(\sigma,K)\ess st(\tau,K)$;
  \item[(ii)] If $L\sse K$ is a subcomplex containing $\sigma$, $st(\sigma,L)\sse st(\sigma,K)$;
  \item[(iii)] For a face $\sigma=\tau\cup\{x\}$, $st(\sigma,K)=st(x,st(\tau,K))$.
  \item[(iv)] $st(\tau)\sse st(\sigma)$ if and only if $\{\text{maximal simplices in }K\text{ that contain }\tau\}$\\  $\sse\{\text{maximal simplices in }K\text{ that contain }\sigma\}$.
\item[(v)] $\sigma\in st(\tau)\Leftrightarrow \tau\in st(\sigma)$.
\item[(vi)] If $\sigma\cup\tau$ is a face of $K$, $st(\sigma)^\circ\cap st(\tau)^\circ=st(\sigma\cup\tau)^\circ$.
  \end{itemize}
\end{Lemma}

\begin{proof}
  $(i)$, $(ii)$ and $(v)$ are obvious. $(iv)$ follows from the fact that a complex is determined by its maximal cells.
$(iii)$ follows from unwrapping the definitions: 
\begin{align}
st(x,st(\tau,K)) &= \{\alpha\in st(\tau,K)|\alpha\cup \{x\} \in st(\tau,K)\}\\
&= \{\alpha\in st(\tau,K)|\alpha\cup \{x\}\cup \tau \in K\}\\
&= \{\alpha\in st(\tau,K)|\alpha\cup \sigma \in K\}\\
&= st(\tau,K)\cap st(\sigma,K)\\
&= st(\sigma,K)
\end{align}
where the last equality follows from (i).
To show $(vi)$, note that for any $\sigma\in K$, $st(\sigma)^\circ$ is the interior of the union of maximal simplices containing $\sigma$. It follows that $st(\sigma\cup\tau)^\circ$ is the interior of the union of maximal simplices containing both $\sigma$ and $\tau$, in other words the maximal simplices in $st(\sigma)\cap st(\tau)$. 
\end{proof}

We will need some properties of the face ideals $I_{st(\sigma)}$ and face rings $R_{st(\sigma)}$ of the subcomplexes $st(\sigma,K)$.

\begin{Lemma}\label{Lemma:face-ideals}
  \begin{enumerate}
  \item If $K_1,K_2$ are subcomplexes of $K$, $I_{K_1}+I_{K_2}=I_{K_1\cap K_2}$ and $I_{K_1}\cap I_{K_2}=I_{K_1\cup K_2}$.
  \item $I_{st(\sigma)} = \langle x_\tau|\tau \in U_\sigma\rangle$.
  \item The minimal primes of $I_K$ are the face ideals $I_{st(\tau)}$ for the maximal simplices $\tau$.
  \end{enumerate}
\end{Lemma}

\begin{proof}
  The first two items follow from the definition of $I_{st(\sigma)}$.
For the last item, observe that $I_{st(\sigma)}$ is clearly prime when $\sigma$ is a maximal simplex, as $I_{st(\sigma)}=\langle x_i|x_i\in U_\sigma\rangle$ and monomial ideals are prime exactly when they are generated by a subset of the variables; observe also that all $I_{st(\sigma)}$ are radical. These observations together with item 1 give the result, as $I_K=\bigcap_{\sigma\sse K\text{ maximal}}I_{st(\sigma)}$.
\end{proof}

We intend to study the ring of differential operators on $R$, so let us define what that is:

\begin{Defn}\label{Defn:diffops}
  The ring $D_k(R)$ of $k$-linear differential operators on a $k$-algebra $R$ is defined inductively by
\[
D_k(R)=\bigcup_{n\ge 0}D_k^n(R)
\]
where $D_k^0(R)=R$ and for $n>0$, $D_k^n(R):=\{\phi\in End_k(R)|\forall r\in R: [\phi,r]\in D_k^{n-1}(R)\}$. Elements of $D_k^n(R)\setminus D_k^{n-1}(R)$ are said to have \emph{order $n$}, and there is a natural filtration 
\[
D_k^0(R) \sse D_k^1(R) \sse D_k^2(R) \sse \cdots
\]
on $D_k(R)$ called the \emph{order filtration}.
\end{Defn}

\begin{Defn}\label{Defn:Weyl-alg}
  The \emph{Weyl algebra} in $n$ variables over $k$ is the ring of differential operators on the polynomial ring $k[x_1,\ldots,x_n]$. It is generated as an $R$-algebra by the \emph{divided power operators} $\dd_i^{(a)}=\frac{1}{a!} \frac{\dd^a}{\dd x_i^a}$, 
with the relations $[x_i,x_j]=[\dd_i^{(a)},\dd_j^{(b)}]=0$ for $i\neq j$, $\dd_i^{(a)}\dd_i^{(b)}=\binom{a+b}{a}\dd_i^{(a+b)}$ and $[\dd_i^{(b)},x_i]=\dd_i^{(b-1)}$ (in particular $[\dd_i,x_i]=1$).
\end{Defn}

\begin{Rem}\label{Rem:divided-power-ops}
We use the divided power operators rather than the usual vector fields $\frac{\dd}{\dd x_i}$ as the latter do not generate the whole ring of differential operators in the case of characteristic $p$; the divided power operators however always generate everything regardless of the characteristic, as they define differential operators on $\ZZ$ and so descend to any commutative ring. In characteristic zero, the derivations $\dd_i$ suffice to generate everything; in characteristic $p$ we need the full set of elements $\dd_i^{p^r}$ for $r\ge 0$, which suffice due to the relation $\dd_i^{(a)}\dd_i^{(b)}=\binom{a+b}{a}\dd_i^{(a+b)}$.
\end{Rem}

In the following, $k$ will always be fixed, so we will omit it from the notation and write simply $D(R)$. Elements of $k$ will be referred to as \emph{constants}. One easily verifies that an element $x^a\dd^{(b)}$ in the Weyl algebra has order $|b|$.

%
\section{The two-sided ideals of $D(R)$}\label{Section:ideals}
%

When $R=R_K$ is a face ring, there exist several descriptions of $D(R)$ in the literature, see \cite{Tripp}, \cite{Eriksson} and \cite{Traves}. We wish to give a description of the two-sided ideals of $D(R)$ in terms of the combinatorics of $K$; for our purposes, the following description due to Traves (\cite{Traves}) is the most convenient.

\begin{Thm}\label{Thm:Traves}
Let $k$ be a commutative domain, and $R = k[X]/J$ a reduced monomial ring. An element $x^a\dd^{(b)}=\prod_{i}x_i^{a_i}\dd_i^{(b_i)}$ of the Weyl algebra over $k$ is in $D(R)$ if and only if for each minimal prime $\p$ of $R$, we have either $x^a\in \p$ or $x^b \not\in \p$. $D(R)$ is generated as a k-algebra by these elements, and they form a free basis of $D(R)$ as a left $k$-module.
\end{Thm}

\begin{Ex}\label{Ex:empty-2-simplex}
  Let $R=k[x_1,x_2,x_3]/(x_1x_2x_3)$. The associated simplicial complex $K$ is the boundary of a 2-simplex. Then by \ref{Thm:Traves}, $D(R)=R\langle x_i^{a_i}\dd_i^{(b_i)}|a_i,b_i\in \NN\rangle$. 
\end{Ex}

\begin{Ex}\label{Ex:three-1-simplices}
   Let $R=k[x_1,x_2,x_3,x_4]/I$ where $I=(x_1x_3,x_1x_4,x_2x_4)$. The associated complex $K$ is a chain of three 1-simplices, connected in order $x_1,x_2,x_3,x_4$. Theorem \ref{Thm:Traves} gives $D(R)=R\langle x_1^a\dd_1^{(b)},x_2^a\dd_2^{(b)},x_3^a\dd_3^{(b)},x_4^a\dd_4^{(b)},x_1^a\dd_2^{(b)},x_4^a\dd_3^{(b)}\rangle$ (for $a,b>0$). 
\end{Ex}

Note that in both examples, generators of the form $x_i^a\dd_i^{(b)}$ appear; it is not hard to see that such ``toric'' operators are always in $D(R)$. In \ref{Ex:three-1-simplices}, we also have generators of the form e.g. $x_i^a\dd_j^{(b)}$ (where $i\neq j$). To understand when this happens, we may give a somewhat more geometric formulation of \ref{Thm:Traves}: 

\begin{Prop}\label{Prop:xdy}
Let $K$ be a simplicial complex and $R=R_K$ its face ring. Also let $x^a=\prod x_i^{a_i}, x^b=\prod x_j^{b_j}$ be such that $supp(x^a)=st(\sigma)$ and $supp(x^b)=st(\tau)$, for some $\sigma,\tau\in K$. Then $x^a\dd^{(b)}=\prod_{i}x_i^{a_i}\dd_i^{(b_i)}$ is in $D(R)$ if and only if $st(\sigma)\sse st(\tau)$.
\end{Prop}

\begin{proof}
  Let $P_{x^a}$ denote the set of minimal primes in $R$ that contain $x^a$, and $P_{\neg x^a}$ the set of minimal primes that does not contain $x^a$. Clearly, $P_{x^a}\cup P_{\neg x^a}$ is equal to the set of minimal primes in $R$; denote this by $P$. Recalling from \ref{Lemma:face-ideals} that the minimal primes of $R$ are the face ideals $I_{st(\alpha)}$ for maximal simplices $\alpha$, we can reformulate these definitions: $P_{x^a}$ is the set of ideals $I_{st(\alpha)}$ such that $\alpha$ is maximal and $x^a\in I_{st(\alpha)}$, in other words those ideals $I_{st(\alpha)}$ such that $\alpha$ is maximal and $\alpha\in U_\sigma$; and $P_{\not x^a}$ is the set of ideals $I_{st(\alpha)}$ with $\alpha$ maximal and contained in $st(\sigma)$. Again using \ref{Lemma:face-ideals}, the ideal $I_{st(\sigma)}$ defining $st(\sigma)$ is equal to the intersection of all ideals in $P_{\neg x^a}$. Unwrapping definitions, we get
  \begin{align*}
    st(\sigma)\sse st(\tau) & \Leftrightarrow I_{st(\sigma)}\ess I_{st(\tau)}\\
& \Leftrightarrow P_{\neg x^a}\ess P_{\neg x^b}\\
& \Leftrightarrow P_{x^a}\sse P_{x^b}.
  \end{align*}
Putting this together with \ref{Thm:Traves}, we have
  \begin{align*}
    x^a\dd^{(b)}\in D(R) & \Leftrightarrow \forall \p\in P: x^a\in \p \vee x^b\not\in\p\\
& \Leftrightarrow \forall \p\in P: \p\in P_{x^a}\vee \p\in P_{\neg x^b}\\
& \Leftrightarrow P = P_{x^a}\cup P_{\neg x^b}\\
& \Leftrightarrow P_{x^a}\sse P_{x^b} \vee P_{\neg x^b}\sse P_{\neg x^a}\text{ (and these are equivalent)}\\
& \Leftrightarrow st(\sigma)\sse st(\tau).
  \end{align*}
\end{proof}

\begin{Ex}\label{Ex:three-2-simplices}
  Let $R=k[x_1,x_2,x_3,x_4,x_5]/(x_1x_3,x_1x_4,x_2x_4)$, the associated $K$ is three 2-simplices $\{x_1,x_2,x_5\},\{x_2,x_3,x_5\},\{x_3,x_4,x_5\}$ glued along the edges $\{x_2,x_5\}$ and $\{x_3,x_5\}$; $x_5$ is a common vertex to all faces. Note that this makes $K$ a simplicial join of $\{x_5\}$ with the complex from Example \ref{Ex:three-1-simplices}. Looking at the closed stars of the faces, we see that 
\[
st(x_1)\sse st(x_2) \sse st(x_5) \ess st(x_3) \ess st(x_4).
\]
As $st(x_1)=st(\{x_1,x_2\})$, $st(x_4)=st(\{x_4,x_3\})$ and for any face $\sigma$, $st(\sigma)=st(\sigma\cup x_5)$ this accounts for all the stars. From this we should by \ref{Prop:xdy} have the ``toric'' generators $x_i^a\dd_i^{(b)}$, and also $x_1^a\dd_2^{(b)},x_1^a\dd_5^{(b)},x_1^a\dd_2^{(b)}\dd_5^{(c)},x_2^a\dd_5^{(b)}$ and the same with $x_1$ and $x_2$ replaced by $x_4$ and $x_3$ respectively (by symmetry). In fact, $st(x_5)=st(\varnothing)=K$, so we should also have $\dd_5^{(a)}=1\cdot\dd_5^{(a)}$ and the description is somewhat redundant. 
\end{Ex}

From \ref{Prop:xdy} we deduce the following very useful criterion.

\begin{Cor}\label{Cor:ideal-containment}
  $\langle x_\tau\rangle \sse \langle x_\sigma\rangle$ if and only if $st(\tau)\sse st(\sigma)$.
\end{Cor}

\begin{proof}
If $st(\tau)\sse st(\sigma)$, it follows from \ref{Prop:xdy} that $x_\tau\dd_\sigma=x_\tau\prod_{i:x_i\in\sigma}\dd_i$ is in $D(R)$. Now observe that $[\cdots [x_\tau\dd_\sigma,x_{i_1}],\cdots,x_{i_r}]=x_\tau$ (where $x_\sigma=\prod_{1\le j\le r} x_{i_j}$), so we have $x_\tau\in\bigcap_{i:x_i\in\sigma}\langle x_i\rangle = \langle x_\sigma\rangle$. 

To show the reverse implication, note that by definition of $I_{st(\sigma)}$, we have $I_{st(\sigma)}\cap\langle x_\sigma\rangle=\langle 0 \rangle$. If now $st(\tau)\not\sse st(\sigma)$, it follows that $\tau\in U_\sigma$, so $x_\tau\in I_{st(\sigma)}$, which finally implies $x_\tau\not\in\langle x_\sigma\rangle$.
\end{proof}

The following very useful result is surprising. 

\begin{Thm}\label{Thm:monomial-ideals}
  Any proper two-sided ideal in $D(R)$ is generated by reduced monomials in the ``ordinary'' variables $x_1,\ldots,x_n$.
\end{Thm}

\begin{proof}
The proof is in three parts:
\begin{enumerate}
\item The ideal $\langle \sum_{(a,b)\in S} c_{ab}x^a\dd^{(b)}\rangle$ (for some index set $S\sse \NN^{2n}$) is equal to the ideal $\langle x^a\dd^{(b)}|(a,b)\in S\rangle$;
\item the ideal $\langle x^a \rangle$ is equal to the ideal $\langle \prod_{a_i\neq 0}x_i \rangle$;
\item the ideal $\langle x^a\dd^{(b)}\rangle$ is equal to the ideal $\langle \prod_{a_i\neq 0}x_i \rangle$.

\end{enumerate}
We will make heavy use of the fact that for any two-sided ideal $I$ and any element $\phi\in D(R)$, the set of commutators $[\phi,I]$ is contained in $I$.

For the first part, recall that we have two natural concepts of grading on the Weyl algebra, that descend to $D(R)$. First, the natural $\ZZ^n$-grading on the Weyl algebra given by the \emph{degree}
\[
deg(x^a\dd^{(b)})=(a_1-b_1,\ldots,a_n-b_n),
\]
which induces a grading on $D(R)$; second we have the $\NN^n$-grading given by the \emph{order}
\[
ord(x^a\dd^{(b)})=(b_1,\ldots,b_n).
\]
Note that 
\begin{align*}
  [x_i\dd_i,x^a\dd^{(b)}] =& x_i\dd_ix^a\dd^{(b)}-x^a\dd^{(b)}x_i\dd_i\\
=& x_i(x^a\dd_i+a_ix^{a-1_i})\dd^{(b)}-x^a(x_i\dd^{(b)}+\dd^{(b-1_i)})\dd_i\\
=& x_ix^a\dd_i\dd^{(b)} + a_ix_ix^{a-1_i}\dd^{(b)}-x^ax_i\dd^{(b)}\dd_i-x^a\dd^{(b-1_i)}\dd_i\\
=& a_ix^a\dd^{(b)}-\binom{b_i-1+1}{1}x^a\dd^{(b)}\\
=& (a_i-b_i)x^a\dd^{(b)}
\end{align*}
(in the remainder we omit the proof of such identities to avoid tedium), and in the case of characteristic $p$, if $a_i-b_i=cp^r$, we have $[x_i^{p^r}\dd_i^{(p^r)},x^a\dd^{(b)}]=cx^a\dd^{(b)}$. In other words, the operators $[x_i\dd_i,-]$ (and $[x_i^{p^r}\dd_i^{(p^r)},-]$) give different weight to each degree-graded component. Note also that 
\[
[x^a\dd^{(b)},x_i]\cdot\dd_i = x^a\dd^{(b-1_i)}\dd_i = b_ix^a\dd^{(b)},
\]
and if $b_i=cp^r$, we have $[x^a\dd^{(b)},x_i^{p^r}]\dd_i^{(p^r)}=cx^a\dd^{(b)}$. In other words the operators $[-,x_i]\dd_i$ (and $[-,x_i^{p^r}]\dd^{p^r}$) give different weight to each order-graded component. Putting these together, we can isolate any term $x^a\dd^{(b)}$ by applying a suitable polynomial in the operators $[x_i\dd_i,-]$, $[x_i^{p^r}\dd_i^{(p^r)},-]$, $[-,x_i]\dd_i$ and $[-,x_i^{p^r}]\dd^{p^r}$.


For the second part, we may reduce to a single variable. We separate the cases by characteristic. If $char(k)=p$, we have $[x_i\dd_i^{(p^r)},x_i^{p^r}]=x_i$, so $x_i$ is in the ideal generated by $x_i^{p^r}$; choosing a power of $p$ larger than $a_i$ we have $x_i^{p^r}=x_i^{a_i}\cdot x_i^{p^r-a_i}$ and so $x_i\in \langle x_i^{a_i}\rangle$. If $char(k)=0$, on the other hand, we have
\[
[x_i\dd_i^{(2)},x_i^{a_i}]=a_ix_i^{a_i}\dd_i+\binom{a}{2}x_i^{a_i-1}
\]
and 
\[
[x_i^2\dd_i^{(3)},x_i^{a_i}]=a_ix_i^{a_i+1}\dd_i^{(2)}+\binom{a_i}{2}x_i^{a_i}\dd_i+\binom{a_i}{3}x_i^{a_i-1}.
\]
If $a_i=0,1$ there is nothing to prove, and if $a_i>1$, we can invert $\frac{a_i-1}{2}\binom{a_i}{2}-\binom{a_i}{3} = \frac{1}{12}a_i(a_i^2-1)$ and get
\[
x_i^{a_i-1}=\frac{12}{a_i(a_i^2-1)}\left( \frac{a_i-1}{2}[x_i\dd_i^{(2)},x_i^{a_i}]-[x_i^2\dd_i^{(3)},x_i^{a_i}]+a_ix_i^{a_i}\cdot x_i\dd_i^{2}\right).
\]
This gives $\langle x_i^{a_i-1}\rangle \sse \langle x_i^{a_i}\rangle$ and by iterating this procedure, $\langle x_i\rangle = \langle x_i^{a_i}\rangle$.


For the third part, observe that $[x^n\dd^{(m)},x_j]=x^n\dd^{(m-1_j)}$ (for $j$ such that $m_j\neq 0$) is a valid identity for all $n,m>0$. Iterating this beginning with $n=a,m=b$ gives $\langle x^a\rangle\sse\langle x^a\dd^{(b)}\rangle$. By applying part 2 this becomes $\langle \prod_{a_i\neq 0}x_i \rangle \sse \langle x^a\dd^{(b)}\rangle$. 

To show the reverse implication $\langle x^a\dd^{(b)}\rangle \sse \langle \prod_{a_i\neq 0}x_i \rangle$ we show $\langle x^a\dd^{(b)}\rangle \sse \langle x_i\rangle$ for the two cases $a_i,b_i\neq 0$ and $a_i\neq 0,b_i=0$. For the first case, $x_i^{a_i}\dd_i^{(b_i)}$ is a factor of $x^a\dd^{(b)}$; and applying the above argument we have that $x_i^{a_i}\dd_i^{(b_i)}\in \langle x_i\rangle$; it follows that $x^a\dd^{(b)}\in \langle \prod_{i:a_i,b_i\neq 0}x_i\rangle$. 

For the second case, $a_i\neq 0,b_i=0$, we may assume $a_i=1$, for if $a_i>1$, then clearly $x^a\dd^{(b)}=x_ix^{a-1_i}\dd^{(b)}\in\langle x_i\rangle$. By the previous case, $x_i\dd_i^{(2)}$ is in $\langle x_i\rangle$, and so is $x^{a+1_i}\dd^{(b)}=x_ix^a\dd^{(b)}$; then of course their commutator
\[
[x^{a+1_i}\dd^{(b)},x_i\dd_i^{(2)}]=-(a_i+1)x^{a+1_i}\dd_i\dd^{(b)}-a_ix^a\dd^{(b)}
\]
is also in $\langle x_i\rangle$. Rewriting this (with $a_i=1$ as we have assumed) we get
\[
x^a\dd^{(b)} = [x^{a+1_i}\dd^{(b)},x_i\dd_i^{(2)}]-2x^{a+1_i}\dd_i\dd^{(b)}
\]
and so $x^a\dd^{(b)}\in \langle x_i\rangle$; it follows that $x^a\dd^{(b)}\in \langle \prod_{i:a_i\neq 0,b_i=0}x_i\rangle$. Taking both cases together we have shown that $x^a\dd^{(b)}\in \langle \prod_{i:a_i\neq 0,b_i=0\vee b_i\neq 0}x_i\rangle = \langle \prod_{i:a_i\neq 0}x_i\rangle$.
\end{proof}

We have shown that all ideals in $D(R)$ are generated by reduced monomials $\prod x_i$ in the variables of $R$; the next question is of course which ones? 
Recall that we will not distinguish between the vertices of the simplicial complex $K$ and the variables of the associated face ring $R$, but refer to either by the same name, e.g. $x_i$. We also remind of the notation $x_\sigma=\prod_{x_i\in\sigma}x_i$.

\begin{Thm}\label{Thm:generators}
  Any proper ideal in $D(R)$ is generated by monomials $x_\sigma$ with $\sigma\in K$ such that $st(\sigma)\neq K$.
\end{Thm}

\begin{proof}
  From \ref{Thm:monomial-ideals} it follows that any ideal in $D(R)$ is generated by reduced monomials in the variables $x_i$, and clearly the monomials corresponding to non-faces cannot occur as they are in $I_K$, so what remains are the monomials $x_\sigma$ for $\sigma\in K$. Only those $x_\sigma$ such that $st(\sigma)\neq K$ generate proper ideals, as otherwise we have $st(\sigma)=K$ and by \ref{Prop:xdy} the elements $1\cdot\dd_i$ where $x_i\in\sigma$ are in $D(R)$, as both $1$ and $\dd_i$ are monomials with support contained in $st(\sigma)=K$; if we write $\sigma=\{x_{i_1},\ldots,x_{i_t}\}$, we have $[\dd_{i_1},[\dd_{i_2},[\cdots,[\dd_{i_r},x_\sigma]\cdots]]]=1$ and so $\langle x_\sigma\rangle = \langle 1\rangle = R$.
\end{proof}

This now gives us all the ideals in $D(R)$, as by sums 
of principal ideals $\langle x_\sigma\rangle$ we can make everything. We may however also take a different approach:
Any two-sided ideal in $D(R)$ is the kernel of some ring homomorphism; the combinatorial structure of the associated simplicial complex $K$ gives rise to several such maps. 
An obvious choice for candidate homomorphisms is the localization at an element $x_\sigma$; we will see that the kernels of such maps is another generating set for the lattice of two-sided ideals in $D(R)$. We introduce the notation $\overline{J}$ for the extension to $D(R)$ of an ideal $J\sse R$.

\begin{Thm}\label{Thm:mainthm}
The kernel of the localization map $D(R)\to D(R)[\frac{1}{x_\sigma}]$ is the extension $\overline{I_{st(\sigma)}}$ of the ideal $I_{st(\sigma,K)}\sse R$ to $D(R)$.
\end{Thm}

\begin{proof}
By \ref{Thm:generators} it is enough to examine what happens in the localization to monomials $x_\alpha$ for $\alpha\in K$.  
Assume first that $x_\sigma=x_i$ (in other words, $\sigma$ is a vertex). Inverting $x_i$ has the effect that for any non-face $\beta=\cup x_j$ containing $x_i$, the monomial $\tfrac{x_\beta}{x_i}=\prod_{x_j\in\beta,j\neq i}x_j$ is zero in the localization. It is clear that no other monomials are killed, so what remains after localization are those monomials supported on a face $\tau$ such that $\tau\cup x_i$ is not a non-face, or clearing negations, that $\tau\cup x_i$ is a face in $K$; in other words the remaining monomials are those supported on a face of $st(x_i)$. 

For the general case, note that inverting $x_\sigma=\prod_ix_i$ is the same as inverting each $x_i$ successively, and observing that we have from \ref{Lemma:stars}(iii) that $st(\sigma,K)=st(x_1,st(\sigma\setminus x_1,K))$, we are done by recursion.
\end{proof}

\begin{Thm}\label{Thm:lattice}
  The lattice of two-sided  ideals in $D(R)$ is generated by the 
ideals $\overline{I_{st(\sigma)}}\sse D(R)$. 
\end{Thm}

\begin{proof}
After applying \ref{Thm:generators} the question is whether we can generate any proper ideal $\langle x_\tau\rangle$ by sums and intersections of the ideals $\overline{I_{st(\sigma)}}$. Considering that $\overline{I_{st(\sigma)}}=\langle x_\alpha|\alpha\in U_\sigma\rangle$, we can look at the intersection of all such ideals that contain $x_\tau$:
\begin{align}
  \bigcap_{\sigma:\tau\in U_\sigma}\overline{I_{st(\sigma)}} =& \langle x_\alpha|\alpha\in\bigcap_{\sigma:\tau\in U_\sigma}U_\sigma\rangle\\
=& \langle x_\alpha|\forall \sigma\in K: \tau\in U_\sigma\Rightarrow\alpha\in U_\sigma\rangle\\
=& \langle x_\alpha|\forall \sigma\in K: \alpha\not\in U_\sigma\Rightarrow \tau\not\in U_\sigma\rangle\\
=& \langle x_\alpha|\forall \sigma\in K: \alpha\cup\sigma\in K\Rightarrow \tau\cup\sigma\in K\rangle\\
=& \langle x_\alpha|\forall \sigma\in K: \sigma\in st(\alpha)\Rightarrow \sigma\in st(\tau)\rangle\\
=& \langle x_\alpha|st(\alpha)\sse st(\tau)\rangle\\
=& \langle x_\tau\rangle
\end{align}
where the last step is applying Corollary \ref{Cor:ideal-containment}.
\end{proof}

\begin{Ex}\label{Ex:localization-three-1-simplices}
Consider again the ring from \ref{Ex:three-1-simplices}, $R=k[x_1,x_2,x_3,x_4]/I$ where $I=(x_1x_3,x_1x_4,x_2x_4)$; the associated complex $K$ is a chain of three 1-simplices. Inverting $x_1$ gives us that $x_3$ and $x_4$ go to zero in the localization as $x_3=\frac{1}{x_1} x_1x_3 \in I$, etc; it follows that the generators $x_4^a\dd_3^{(b)}$ are also killed; the kernel of the localization $D(R)\to D(R)[\frac{1}{x_1}]$ is then (using \ref{Thm:monomial-ideals} and \ref{Prop:xdy}) the ideal $(x_3,x_4)$, which is the face ideal of $st(x_1,K)$. Localizing at $x_2$ gives $x_3=\frac{1}{x_2} x_2x_4 = 0$, and the kernel of the localization is indeed equal to the ideal $(x_4)$, the face ideal of $st(x_2,K)$. Proceeding in the same manner for the remaining faces $x_3,x_2,\{x_1,x_2\},\{x_2,x_3\}$, and $\{x_3,x_4\}$, we get as possible kernels the ideals $(x_1),(x_4),(x_1,x_2),(x_3,x_4)$ and $(x_2,x_3)$. By \ref{Prop:xdy} we have $(x_1,x_2)=(x_2)$ and $(x_3,x_4)=(x_3)$; in other words our possible kernels of localization are the ideals $(x_1),(x_2),(x_3)$ and $(x_4)$; in light of \ref{Thm:monomial-ideals} these obviously generate all the  ideals by sums and intersections.
\end{Ex}

Let us round off this section with some applications. In \cite{Traves}, Traves examines the $D(R)$-module structure of $R$ when $k$ is a field, and determines what the (left) $D(R)$-submodules of $R$ are. These are the ideals $I\sse R$ such that $D(R)\bullet I = I$, so we follow Traves' terminology and call such a submodule a \emph{$D(R)$-stable ideal}. The reason for restricting $k$ to be a field is that elements of $D_k(R)$ are $k$-linear endomorphisms of $R$, so any ideal of $k$ extends to a $D_k(R)$-submodule of $R$. 

\begin{Thm}[Traves]\label{Thm:Traves-D-mod-struct-of-R}
  When $k$ is a field, the $D_k(R)$-submodules of the reduced monomial ring $R$ are exactly the ideals given by intersections of sums of minimal primes of $R$. 
\end{Thm}

Based on our results about the ideal structure of $D(R)$, we can give a new proof of this result. We denote the module action of $D(R)$ by $\bullet$ (e.g. $D(R)\bullet I$) and the product in $D(R)$ by $\cdot$ (e.g. $D(R)\cdot I$). We prove the result by means of a general fact which to our knowledge is previously unknown. 

\begin{Prop}\label{Prop:D-stable}
  Let $k$ be a field and $R$ be a $k$-algebra. An ideal $J\sse R$ is $D(R)$-stable if and only if $J=\overline{J}\cap R$, where $\overline{J}$ denotes the extension of $J$ to $D(R)$.
\end{Prop}

\begin{proof}
Observe first that $R$ is isomorphic as a $D(R)$-module to $D(R)/D^{>0}(R)$, the quotient by the left ideal of positive order elements; we can see this by writing $D(R) = D^0(R)+D^{>0}(R) = R+D^{>0}(R)$, as $R=D^0(R)$. In other words, if $S\sse R$ is a subset, then under this isomorphism $D(R)\bullet S = D(R)\cdot S + D^{>0}(R)$. Further, if $J\in D(R)$ is some subset, then 
\begin{align*}
J\cdot D(R)+D^{>0}(R) =& J\cdot (D(R)^0+D^{>0}(R)) + D^{>0}(R)\\
=& J\cdot R + D^{>0}(R).
\end{align*}
Now, if $I\sse R$ is an ideal, the extension of $I$ to $D(R)$ is $\overline{I}=D(R)\cdot I\cdot D(R)$, so we have
\begin{align*}
  \overline{I}+D^{>0}(R) =& D(R)\cdot I\cdot D(R)+D^{>0}(R)\\
=& D(R)\cdot I + D^{>0}(R)\\
=& D(R)\bullet I.
\end{align*}
A $D$-stable ideal is an ideal $I\sse R$ such that $D(R)\bullet I = I$, so it follows that the $D$-stable ideals are exactly those such that $\overline{I}+D^{>0}(R) = I$. 

It remains to show that for an ideal $J\sse D(R)$, $J+D^{>0}(R) = J\cap R$. Let $f\in J$ be some element, and write it as the sum $f=f_0+f_1+\cdots+f_{ord(f)}$ where $f_i$ are the terms of order $i$; it then follows from \ref{Thm:monomial-ideals} that also each $f_i\in J$. Reducing modulo $D^{>0}(R)$ we get $J+D^{>0}(R) = \{f_0|f\in J\}$, and restricting to the homogenous elements of order zero we have $J\cap R = J\cap D^0(R) = \{f\in J|f=f_0\}$; these sets clearly are equal.
\end{proof}

\begin{Thm}\label{Thm:D-stable-I-sigma}
  The $D(R)$-stable ideals of $R$ are those generated by sums and intersections of the ideals $I_{st(\sigma)}$ for $\sigma\in K$.
\end{Thm}

\begin{proof}
  As we have shown (\ref{Thm:generators}, \ref{Thm:lattice}) that any ideal of $D(R)$ is an extension of an ideal of $R$, we only have to restrict these to $R$ to recover the $D(R)$-stable ideals. Theorem \ref{Thm:lattice} tells us that the lattice of ideals in $D(R)$ is generated by sums and intersections of ideals $\overline{I_{st(\sigma)}}$, and it is easy to see that 
$\overline{I_{st(\sigma)}}\cap R = I_{st(\sigma)}$:
Indeed, the only possible problem is that in $D(R)$, $\langle x_\alpha\rangle\sse\langle x_\beta\rangle$ if and only if $st(\alpha)\sse st(\beta)$, and this may cause additional monomials not in $I$ to appear in $\overline{I}\cap R$. For $I_{st(\sigma)}$ however, this does not happen. Consider that $I_{st(\sigma)}=\langle x_\tau|\tau\in U_\sigma\rangle$ and $\overline{I_{st(\sigma)}}\cap R=\langle x_\tau|\tau\in U_\sigma\rangle=\langle x_\alpha|\exists \tau\in U_\sigma: st(\alpha)\sse st(\tau)\rangle$. In other words, we need to check if there are faces $\tau\in U_\sigma$ and $\alpha \in st(\sigma)$ such that $st(\alpha)\sse st(\tau)$, as then $x_\alpha$ would be in $\overline{I_{st(\sigma)}}\cap R$, but not in $I_{st(\sigma)}$. This is impossible, however: by \ref{Lemma:stars}$(v)$, $\alpha \in st(\sigma)$ if and only if $\sigma\in st(\alpha)$, and if $st(\alpha)\sse st(\tau)$, we have $\sigma\in st(\tau)$, which again by \ref{Lemma:stars}$(v)$ gives $\tau\in st(\sigma)$, which contradicts the assumption $\tau\in U_\sigma$.
\end{proof}

To recover \ref{Thm:Traves-D-mod-struct-of-R}, recall that by \ref{Lemma:face-ideals}, the minimal primes are exactly the face ideals of the maximal faces of $K$, and any $I_{st(\sigma)}$ is the intersection of the face ideals of the maximal faces of $st(\sigma)$.

\begin{Rem}
  Recall that the partially ordered set of two-sided ideals of $D(R)$ (or bijectively, the $D(R)$-stable ideals of $R$) is in order-reversing bijection with the partially ordered set of closed stars of $K$. This partially ordered set can be completed to a simplicial complex ($\widetilde{K}$, say), homotopic to the nerve of the cover of $K$ by open stars. The results about two-sided ideals of $D(R)$ and $D(R)$-stable ideals of $R$ imply that subcomplexes $L$ of $K$ such that $I_L$ is $D(R)$-stable or $\overline{I_L}$ is a two-sided ideal of $D(R)$ are exactly those that are unions of intersections of closed stars; in other words the complex $\widetilde{K}$ classifies such subcomplexes. This interesting connection is perhaps worthy of further study.
\end{Rem}

%
\section{Characteristic $p$}\label{Sec:char-p}
%

The constructions in the previous section are independent of the characteristic of $k$, and so solve the problem of finding the two-sided ideal structure of $D(R)$. In characteristic $p$ however, there is a qualitatively different construction of $D(R)$, which perhaps offers more interesting possibilities for generalization. From here on, we assume $k$ is a field of characteristic $p$.

The major tool when working in characteristic $p$ is the Frobenius automorphism of $k$, given by $x\mapsto x^p$. This induces an endomorphism $F:R\to R$ given by $F(f) = f^p$, and the image $F(R)$ is the subring $R^p \sse R$ of $p$'th powers; as $R$ is reduced $F$ is also an isomorphism onto its image. Any $R$-module $M$ gets a new $R$-module structure through the pullback by the Frobenius map, namely $F_*M$ is equal to $M$ as an abelian group, but has $R$-module structure given by $f\cdot m = f^p m$. This is equivalent to considering $M$ as an $R^p$-module, as the maps $F:R\to R$ and $R^p\inj R$ both are injections with image $R^p$. We will have need for considering also iterates of $F$, so if we let $q=p^r$ we write $F^r:R\to R$ or $R^q=R^{p^r}\sse R$. For our purposes in examining $D(R)$, it will be most convenient to use the description in terms of the subrings $R^q$, as we will see. 

Considering the behaviour $R$ itself as an $R^p$-module gives rise to several classifying properties of the ring $R$. We will simply recall the definitions of the particular properties that are relevant for us, other such properties and further details may be found in \cite{Smith-vdB}. If $R$ is finitely generated as an $R^p$-module, we say that $R$ is \emph{$F$-finite}; if $R$ is $F$-finite and the map $R^p\inj R$ splits as a map of $R^p$-modules we say $R$ is \emph{$F$-split}; if $F^r_*R\simeq M_1^r\oplus\cdots\oplus M_{n(r)}^q$ as an $R$-module and the set of isomorphism classes $\{[M_i^r]|r\in \NN,1\le i\le n(r)\}$ of modules appearing in such a decomposition for some $r$ is finite, we say that $R$ has \emph{finite $F$-representation type}, or \emph{FFRT}.

For our purposes, the key property of face rings $R_K$ in this respect is that they are $F$-split and have FFRT. Even better, we can give a concrete decomposition of $R$ as an $R^q$-module:

\begin{Lemma}\label{Lemma:Fr-decomp}
  As an $R^q$-module, $R$ is isomorphic to $\bigoplus_{st(\sigma)\sse K}(R^q_{st(\sigma)})^{m_{st(\sigma)}(q)}$, where $m_{st(\sigma)}(q)=\sum_{\alpha:st(\alpha)= st(\sigma)}(q-1)^{(\dim(\alpha)+1)}$.
\end{Lemma}

Note that the direct sum runs over those subcomplexes of $K$ that is the star of some simplex. 

\begin{proof}
  As we have $R\simeq R^p\oplus R^px_1 \oplus \cdots \oplus R^p x_1^{p-1}\cdots x_n^{p-1}$ (where only the appropriate monomials appear), this expresses $R$ as an $R^p$-module. 
We can rewrite this using $R^p\cdot x^\alpha \simeq R^p/Ann_{R^p}(x^\alpha)$, and observing that as the monomials $x^\alpha$ that appear in the decomposition are those supported on a face $supp(\alpha)=:\sigma$, and that the annihilator of $x^\alpha$ is the face ideal of the complex $st(\sigma,K)$, we get the decomposition $R=\bigoplus_{\sigma\in K}(R^p_{st(\sigma)})^{m_{st(\sigma)}}$, where $R^p_{st(\sigma)}$ is the ($p$'th power) face ring of $st(\sigma)$ and by simply counting monomials we have $m_{st(\sigma)}=\sum_{\alpha:st(\alpha)= st(\sigma)}(p-1)^{(\dim(\alpha)+1)}$ (using the convention that $\dim(\varnothing)=-1$). Iterating the same construction, we get $R=\bigoplus_{\sigma\in K}(R^q_{st(\sigma)})^{m_{st(\sigma)}(q)}$, where $m_{st(\sigma)}(q)=(q-1)^{(\dim(\sigma)+1)}$.
\end{proof}

Let us make use of this to compute some invariants of $R$ that only make sense in characteristic $p$, namely the \emph{Hilbert-Kunz function} and the \emph{Hilbert-Kunz multiplicity}. This invariant was introduced by Kunz \cite{Kunz} for local rings, and extended to graded rings by Conca \cite{Conca}; see also \cite{Huneke} and \cite{Monsky}.

\begin{Defn}\label{Defn:HK-mult}
Let $R$ be a local ring with maximal ideal $\mathfrak{m}$, or a graded ring with homogenous maximal ideal $\mathfrak{m}$, over a field $k$ of characteristic $p$, and let $q=p^r$. The \emph{Hilbert-Kunz function} of a ring $R$ is the function
\[
HK_R(q)=l(R/\mathfrak{m}^{[q]})
\]
where $I^{[q]}$ is the ideal generated by $q$'th powers of elements in the ideal $I$. The \emph{Hilbert-Kunz multiplicity} is the number
\[
e_{HK}(R) = \lim_{q\to\infty} \frac{HK_R(q)}{q^{\dim R}},
\]
in other words the leading coefficient of $HK_R(q)$.
\end{Defn}

The Hilbert-Kunz function gives a measure of singularity of $R$, roughly speaking higher multiplicities correspond to worse singularities. It is a theorem of Kunz that $HK_R(q)=q^{\dim R}$ if and only if $R$ is regular (see \cite{Kunz}), so if $R$ is regular, $e_{HK}(R)=1$. The converse holds for \emph{unmixed} rings, but not in general, and in particular not for face rings. The following is equivalent to Remark 2.2 in \cite{Conca}, though we prove it in a different way.

\begin{Prop}\label{Prop:HK-of-R}
  Let $R_K$ be a face ring, then $HK_R(q) = \sum_{i=-1}^{\dim(R)-1}f_i(q-1)^{i+1}$, where $f_i$ is the number of $i$-simplices in $K$, so $(f_{-1},\ldots,f_{\dim(R)-1})$ is the $f$-vector of $K$ (we recall the usual convention $\dim(\varnothing)=-1$, so $f_{-1}=1$). In particular, $e_{HK}(R_K)=f_{\dim K}$, the number of top-dimensional faces of $K$.
\end{Prop}

\begin{proof}
  The number of indecomposable summands of $R$ as an $R^q$-module is $\sum_{\sigma\in K}(q-1)^{dim(\sigma)+1}$ by \ref{Lemma:Fr-decomp}. By simply rearranging the sum, this is equal to $\sum_{i=-1}^{\dim(R)-1}f_i(q-1)^{i+1}$. The claim now follows from the fact that none of the generators of these summands are in $\mathfrak{m^{[q]}}=\langle x_1^q,\ldots ,x_n^q\rangle$, so the number of summands in the splitting of $R$ is the same as the length of $R/\mathfrak{m^{[q]}}$.
\end{proof}

The promised different construction of $D(R)$ is due to Yekutieli \cite{Yekutieli}. We omit the proof here, but mention that in addition to \cite{Yekutieli}, the reader can find an excellent exposition in \cite{Smith-vdB}.

\begin{Prop}\label{Prop:Yekutieli}
  $D_k(R)\simeq \bigcup_q End_{R^q}(R)$, where $q=p^r, r\in \NN$ and $R^q$ is the subring of $q$-th powers.
\end{Prop}

Let us now give the summands appearing in \ref{Lemma:Fr-decomp} a more convenient notation, and define $M_{st(\sigma)}^q:=(R^q_{st(\sigma)})^{m_{st(\sigma)}(q)}$. It follows from \ref{Lemma:Fr-decomp} that 
\[
End_{R^q}(R)\simeq \bigoplus_{st(\sigma),st(\tau)\sse K}Hom_{R^q}(M_{st(\sigma)}^q,M_{st(\tau)}^q).
\]
As each $M_{st(\sigma)}^q$ is generated as an $R^q$-module by monomials of degree in each variable up to $q-1$, we can see that as an $R^{pq}$-module it is contained in $\bigoplus_{st(\alpha)\sse st(\sigma)}M_{st(\alpha)}^{pq}$, because the elements of $M_{st(\sigma)}^q$ contain monomials of degree larger than $q-1$, which have support on smaller stars (recall that as $q=p^r$, $pq=p^{r+1}$). In particular this implies the following:

\begin{Lemma}\label{Lemma:q-to-pq}
  $Hom_{R^q}(M_{st(\sigma)}^q,M_{st(\tau)}^q) \sse \bigoplus_{st(\alpha)\sse st(\sigma),st(\beta)\sse st(\tau)}Hom_{R^{pq}}(M_{st(\alpha)}^{pq},M_{st(\beta)}^{pq})$.
\end{Lemma}

This lets us think of elements $\phi\in End_{R^q}(R)$ as block matrices with each block having entries in some $R^q/I_{st(\sigma)}$; it is vital to remember that this means that the entries have degree equal to a multiple of $q$.

\begin{Defn}\label{Defn:J-sigma-tau}
 Let $J_q(st(\alpha), st(\beta))$ denote the ideal in $D(R)$ generated by the elements of $Hom_{R^q}(M_{st(\alpha)}^q,M_{st(\beta)}^q)$, and let $J(st(\alpha), st(\beta)) := \sum_q J_q(st(\alpha), st(\beta))$. For convenience we denote $J(st(\sigma),st(\sigma))$ by simply $J(st(\sigma))$. 
\end{Defn}

The following result is essentially the same as \ref{Thm:monomial-ideals} in a different guise.

\begin{Prop}\label{Prop:J-sigma-principal}
  Assume $st(\sigma)\ess st(\tau)$, and let $\phi\in Hom_{R^q}(M_{st(\sigma)}^q,M_{st(\tau)}^q)$ be a nonzero element. Then $\langle \phi\rangle$, the ideal in $D(R)$ generated by $\phi$, is equal to the ideal $J(st(\tau))$. Furthermore, we have that $J(st(\tau)\sse J(st(\sigma))$.
\end{Prop}

\begin{proof}
Clearly, $J(st(\tau))$ is generated by the identity maps $id_{st(\tau)}^q:M_{st(\tau)}^q\to M_{st(\tau)}^q$ (for each $q$), so it suffices to show that these are in $\langle \phi \rangle$. 

Recall that any element of $End_{R^{q}}(R)$ has entries with degree a multiple of $q$. We claim that for $s>q$ a sufficiently large power of $p$, $\phi$ considered as an element of $End_{R^{s}}(R)$ will have at least some constant entries in each block $Hom_{R^s}(M_{st(\sigma)}^s,M_{st(\tau)}^s)$. To see this, suppose $\phi$ (as an element of $End_{R^{q}}(R)$) has an entry $x_i^q$ in a block $Hom_{R^q}(R^q\cdot x^a,R^q\cdot x^b)$ (with all $0\le a_j,b_j<q$), in other words $\phi(x^{a+cq})=x^{a+(c+1_i)q+b}$. It follows from \ref{Lemma:q-to-pq} that this block has image in $End_{R^{pq}}(R)$ contained in $\bigoplus_{0\le c,d<p} Hom_{R^{pq}}(R^{pq}\cdot x^{a+cq},R^{pq}\cdot x^{b+dq})$, and as $\phi(x^{a+cq}) = x^{a+(c+1_i)q+b} = x^{a+cq+(b+1_iq)}$ this yields the entry 1 in the blocks $Hom_{R^{pq}}(R^{pq}\cdot x^{a+cq},R^{pq}\cdot x^{b+1_iq})$. In similar fashion an entry with degree $nq$ will yield constant entries somewhere when considered as an $R^s$-linear map for $s>q$ a sufficiently large power of $p$. 

Now let $s$ be such a sufficiently large power of $p$, and consider $\phi$ as an element of $End_{R^{s}}(R)$; by \ref{Lemma:q-to-pq}, $Hom_{R^q}(M_{st(\sigma)}^q,M_{st(\tau)}^q)$ is contained in $\bigoplus_{st(\alpha)\sse st(\sigma),st(\beta)\sse st(\tau)}Hom_{R^{s}}(M_{st(\sigma)}^{s},M_{st(\tau)}^{s})$. We can see that $\phi$, considered as a matrix $(\phi_{ij})$ in $End_{R^{s}}(R)$, will have (among others) some constant entries in each block $End_{R^{s}}(M_{st(\beta)}^{s})$ such that $st(\beta)\sse st(\tau)$. Each of these entries can be ``picked out'' in the following manner: Let $\mathbf{1}_{ii}$ be the matrix in $End_{R^{s}}(R)$ with the appropriate identity map in position $(i,i)$ and zeroes otherwise. It is clear that $\mathbf{1}_{ii}\cdot\phi\cdot\mathbf{1}_{jj}$ is the matrix with entry $\phi_{ij}$ in position $(i,j)$ and zeroes otherwise; we may assume $\phi_{ij}=1$ as it is constant. Applying permutations of $End_{R^{s}}(M_{st(\beta)}^{s})$ (on both sides), we can now place this entry $1$ wherever we want within the matrix block corresponding to $End_{R^{s}}(M_{st(\beta)}^{s})$; taking sums of these we can produce any matrix with constant entries. In particular, we can make $id_{st(\beta)}^{s}$.

Thus, we have that each $id_{st(\beta)}^{s}$ such that $st(\beta)\sse st(\tau)$ is in $\langle\phi\rangle$, and in the same way any such $id_{st(\beta)}^{t}$ for $t>s$ any larger power of $p$. To recreate $id_{st(\tau)}^{t}$ for smaller powers $t<s$ we observe that those maps, considered as elements of $End_{R^{s}}(R)$, are in $\bigoplus_{st(\beta)\sse st(\tau)} End_{R^{s}}(M_{st(\beta)})$ and as such are contained in the ideal generated by the identity maps $id_{st(\beta)}^{s}$, in other words contained in $\langle\phi\rangle$. We have shown $J(st(\tau))\sse \langle\phi\rangle$; the opposite inclusion follows from the observation that $\phi = id_{st(\tau)}^q\circ \phi$, and so $\phi\in J(st(\tau))$.

The final claim is similar: $\phi = \phi\circ id_{st(\sigma)}^q$, and so $\phi\in J(st(\sigma))$.
\end{proof}

\begin{Prop}\label{Prop:J-sigma-endomorphism}
  The ideal $J(st(\sigma),st(\tau))$ is equal to $J(st(\sigma\cup\tau))$, if $\sigma\cup\tau$ is a face of $K$, and the zero ideal otherwise.
\end{Prop}

\begin{proof}
The module $Hom_{R^q}(M_{st(\sigma)}^q,M_{st(\tau)}^q)$ has support $st(\sigma)^\circ\cap st(\tau)^\circ$. From \ref{Lemma:stars}$(vi)$ it follows that this is $st(\sigma\cup\tau)^\circ$, if $\sigma\cup\tau\in K$. 

If $\sigma\cup\tau$ is a non-face, $st(\sigma)\cap st(\tau)$ does not contain any maximal simplices, and so the cone on $st(\sigma)\cap st(\tau)$ is not a union of irreducible components of $Spec(R)$, and so is not the closure of the support of any element in $Hom_{R^q}(M_{st(\sigma)}^q,M_{st(\tau)}^q)$, so this must be the zero module. It follows that $J(st(\sigma),st(\tau))$ is the zero ideal.

For the case when $\sigma\cup\tau$ is a face of $K$, recall that by Lemma \ref{Lemma:q-to-pq}, 
\[
Hom_{R^q}(M_{st(\sigma)}^q,M_{st(\tau)}^q)\sse \bigoplus_{st(\alpha)\sse st(\sigma),st(\beta)\sse st(\tau)}Hom_{R^{pq}}(M_{st(\alpha)}^{pq},M_{st(\beta)}^{pq}).
\]
In particular, there will be entries in the block $Hom_{R^{pq}}(M_{st(\sigma\cup\tau)}^{pq},M_{st(\sigma\cup\tau)}^{pq})$, so by \ref{Prop:J-sigma-principal} we have that $J(st(\sigma\cup\tau))\sse J(st(\sigma),st(\tau))$.

For the converse, note that as an $R^q$-module, 
\[
Hom_{R^q}(M_{st(\sigma)}^q,M_{st(\tau)}^q)\simeq \big((I_{st(\tau)}^q:I_{st(\sigma)}^q)/I_{st(\tau)}^q\big)^{m_{st(\sigma)}(q) \times m_{st(\tau)}(q)}
\]
(where $I^q$ is the restriction of $I\sse R$ to $R^q$). Any element of $Hom_{R^q}(M_{st(\alpha)}^q,M_{st(\beta)}^q)$ has, as a matrix, entries with degree (in each variable) a multiple of $q$, with constant (nonzero) entries only when $st(\beta)\sse st(\alpha)$, as then $(I_{st(\beta)}^q:I_{st(\alpha)}^q)$ is the unit ideal in $R^q$ (otherwise it is generated by elements of degree $\ge q$). It follows that elements of the image of $Hom_{R^q}(M_{st(\sigma)}^q,M_{st(\tau)}^q)$ in $End_{R^s}(R)$ for $s>q$ (considered as matrices) have entries with degree some multiple of $s$, with constant (nonzero) entries only in those blocks $Hom_{R^s}(M_{st(\alpha)}^s,M_{st(\beta)}^s)$ with $st(\beta)\sse st(\alpha)$. In the direct limit, these elements become infinite matrices with entries in $k$, in other words there can only be nonzero entries in those blocks corresponding to $st(\beta)\sse st(\alpha)$ (any nonzero entry in a different block must have infinite degree, which is impossible). This implies that $J(st(\sigma),st(\tau))$ is contained in $\sum_{st(\sigma)\ess st(\alpha) \ess st(\beta)\sse st(\tau)}J(st(\alpha),st(\beta))$, which by \ref{Prop:J-sigma-principal} is equal to $\sum_{st(\sigma)\ess st(\beta)\sse st(\tau)}J(st(\beta))=J(st(\sigma\cup\tau))$ and we are done.
\end{proof}

\begin{Thm}\label{Thm:char-p-lattice}
  The ideals $J(st(\sigma))$ generate the lattice of ideals in $D(R)$ by sums and intersections.
\end{Thm}

\begin{proof}
Let $I$ be an ideal in $D(R)$; it is of course true in general that $I=\sum_{\phi\in I} \langle \phi\rangle$. By \ref{Prop:J-sigma-principal} and \ref{Prop:J-sigma-endomorphism} this is equal to $\sum J(st(\sigma))$, where the sum goes over all $\sigma\in K$ such that $I$ contains elements from some $Hom_{R^q}(M_{st(\alpha)}^q,M_{st(\sigma)}^q)$. 

Finally, the intersection $J(st(\sigma))\cap J(st(\tau))$ contains elements in those $End_{R^q}(M_{st(\alpha)}^q)$ with $st(\alpha)\sse st(\sigma)\cap st(\tau)$; the maximal such star is $st(\sigma\cup\tau)$ if $\sigma\cup\tau$ is a face of $K$, and if $\sigma\cup\tau$ is not a face, there are no such $\alpha$; in other words $J(st(\sigma))\cap J(st(\tau)) = J(st(\sigma\cup\tau))$.
\end{proof}

We have now given two essentially different descriptions of the ideals of $D(R)$, and we may wonder how to translate between the two languages. This is not too hard, as the obvious suggestion turns out to be true.

\begin{Thm}\label{Thm:translation}
  The ideal $J(st(\sigma))$ is equal to the ideal $\langle x_\sigma\rangle$.
\end{Thm}

\begin{proof}
  It follows from \ref{Prop:J-sigma-principal} and \ref{Prop:J-sigma-endomorphism} that $J(st(\sigma)) = \bigoplus_{q>0,st(\beta)\sse st(\sigma)}Hom_{R^q}(M^q_{st(\alpha)},M_{st(\beta)}^q)$, in other words all the endomorphisms with support contained in $st(\sigma)$. We can think of $x_\sigma$ as an endomorphism of $R$, given by $f\mapsto f x_\sigma$, and considering that whatever element $f$ we choose, $f x_\sigma$ has support contained in $st(\sigma)$. This means that the endomorphism $x_\sigma$ is in $J(st(\sigma))$ and not in any larger ideal, and as $x_\sigma(1)=x_\sigma$ has support equal to $st(\sigma)^\circ$, it is not in any smaller ideal $J(st(\tau))$ with $st(\tau)\sse st(\sigma)$. From \ref{Prop:J-sigma-principal} it follows that $x_\sigma$ generates all of $J(st(\sigma))$ and the two ideals are equal.
\end{proof}

\section*{Acknowledgements}

I would like to thank my advisor Rikard Bøgvad for all the usual reasons, and I also thank Anders Björner for some helpful remarks.


\begin{thebibliography}{SVdB97}

\bibitem[Bav10a]{Bavula-0}
V.~V. Bavula, \emph{Generators and defining relations for the ring of
  differential operators on a smooth affine algebraic variety}, Algebr.
  Represent. Theory \textbf{13} (2010), no.~2, 159--187.

\bibitem[Bav10b]{Bavula-p}
\bysame, \emph{Generators and defining relations for the ring of differential
  operators on a smooth affine algebraic variety in prime characteristic}, J.
  Algebra \textbf{323} (2010), no.~4, 1036--1051.

\bibitem[Con96]{Conca}
Aldo Conca, \emph{Hilbert-{K}unz function of monomial ideals and binomial
  hypersurfaces}, Manuscripta Math. \textbf{90} (1996), no.~3, 287--300.

\bibitem[Eri98]{Eriksson}
Anders Eriksson, \emph{The ring of differential operators of a
  {S}tanley-{R}eisner ring}, Comm. Algebra \textbf{26} (1998), no.~12,
  4007--4013.

\bibitem[Hun13]{Huneke}
Craig Huneke, \emph{Hilbert-{K}unz multiplicity and the {F}-signature},
  Commutative algebra, Springer, New York, 2013, pp.~485--525.

\bibitem[Kun69]{Kunz}
Ernst Kunz, \emph{Characterizations of regular local rings for characteristic
  {$p$}}, Amer. J. Math. \textbf{91} (1969), 772--784.

\bibitem[Mon83]{Monsky}
P.~Monsky, \emph{The {H}ilbert-{K}unz function}, Math. Ann. \textbf{263}
  (1983), no.~1, 43--49.

\bibitem[Mus94]{Musson}
Ian~M. Musson, \emph{Differential operators on toric varieties}, J. Pure Appl.
  Algebra \textbf{95} (1994), no.~3, 303--315.

\bibitem[Sai07]{Saito}
Mutsumi Saito, \emph{Primitive ideals of the ring of differential operators on
  an affine toric variety}, Tohoku Math. J. (2) \textbf{59} (2007), no.~1,
  119--144.

\bibitem[SVdB97]{Smith-vdB}
Karen~E. Smith and Michel Van~den Bergh, \emph{Simplicity of rings of
  differential operators in prime characteristic}, Proc. London Math. Soc. (3)
  \textbf{75} (1997), no.~1, 32--62.

\bibitem[Tra99]{Traves}
William~N. Traves, \emph{Differential operators on monomial rings}, J. Pure
  Appl. Algebra \textbf{136} (1999), no.~2, 183--197.

\bibitem[Tri97]{Tripp}
J.~R. Tripp, \emph{Differential operators on {S}tanley-{R}eisner rings}, Trans.
  Amer. Math. Soc. \textbf{349} (1997), no.~6, 2507--2523.

\bibitem[Yek92]{Yekutieli}
Amnon Yekutieli, \emph{An explicit construction of the {G}rothendieck residue
  complex}, Ast\'erisque (1992), no.~208, 127, With an appendix by Pramathanath
  Sastry.

\end{thebibliography}

\providecommand{\bysame}{\leavevmode\hbox to3em{\hrulefill}\thinspace}
\providecommand{\MR}{\relax\ifhmode\unskip\space\fi MR }
\providecommand{\MRhref}[2]{%
  \href{http://www.ams.org/mathscinet-getitem?mr=#1}{#2}
}
\providecommand{\href}[2]{#2}

\end{document}